\documentclass[a4paper,reqno, 11pt]{amsart}
\usepackage[T1]{fontenc}
\usepackage[utf8]{inputenc}

\usepackage{amsmath,amssymb, graphicx, enumerate, enumitem, color}
\usepackage[usenames,dvipsnames,svgnames,table]{xcolor}
\usepackage[foot]{amsaddr}

\usepackage{multirow}
\usepackage{longtable}
\usepackage{booktabs}

\usepackage[a4paper, left=2.5cm, right=2.5cm, top=2.5cm, bottom=2.5cm]{geometry}

\usepackage{hyperref}
\usepackage{xcolor}
\hypersetup{
   pdfauthor={Marcin Briański, Stefan Felsner, Jędrzej Hodor, Piotr Micek},
   pdftitle={Reconfiguring Independent Sets on Interval Graphs},
   colorlinks,
   linkcolor={red!50!black},
   citecolor={blue!50!black},
   urlcolor={blue!80!black}
}

\usepackage[noabbrev,capitalise]{cleveref}

\newtheorem{theorem}{Theorem}
\newtheorem{conjecture}[theorem]{Conjecture}

\newtheorem{remark}[theorem]{Remark}
\newtheorem{proposition}[theorem]{Proposition}

\newtheorem{claim}[theorem]{Claim}

\newtheorem{lemma}[theorem]{Lemma}
\theoremstyle{remark}

\newcommand\restr[2]{{
  \left.\kern-\nulldelimiterspace 
  #1 
  \vphantom{\big|} 
  \right|_{#2} 
  }}

\usepackage{mathtools}
\DeclarePairedDelimiter\set{\{}{\}}

\DeclarePairedDelimiter\paren{(}{)}

\DeclarePairedDelimiter\card{|}{|}

\usepackage{algorithm}
\usepackage[noend]{algpseudocode}
\algdef{SE}[DOWHILE]{Do}{doWhile}{\algorithmicdo}[1]{\algorithmicwhile\ #1}

\algnewcommand{\IIf}[1]{\State\algorithmicif\ #1\ \algorithmicthen}
\algnewcommand{\EndIIf}{\unskip}

\usepackage{todonotes}

\newcommand{\calA}{\mathcal{A}}
\newcommand{\calI}{\mathcal{I}}

\newcommand{\calC}{\mathcal{C}}
\newcommand{\calD}{\mathcal{D}}

\newcommand{\Oh}{\mathcal{O}}

\newcommand{\isr}{\textsf{Independent Set Reconfiguration}}

\DeclareMathOperator\ex{ex}

\let\le\leqslant
\let\ge\geqslant
\let\leq\leqslant
\let\geq\geqslant

\let\epsilon\varepsilon

\let\phi\varphi

 \renewenvironment{enumerate}{\begin{enumorig}[label=\textup{(\roman*)}, noitemsep, topsep=2pt plus 2pt, labelindent=.2em, leftmargin=*, widest=iii]}{\end{enumorig}}

\begin{document}

\title[Reconfiguring independent sets on interval graphs]{Reconfiguring independent sets on interval graphs}

\author[M.~Briański]{Marcin Briański}
\address[M.~Briański, J.\ Hodor, P.\ Micek]{Theoretical Computer Science Department\\
 Faculty of Mathematics and Computer Science, Jagiellonian University, Krak\'ow, Poland}
\email{marcin.brianski@doctoral.uj.edu.pl}

\author[S.~Felsner]{Stefan Felsner}
\address[S.~Felsner]{Institut f\"ur Mathematik\\
 Technische Universit\"at Berlin\\
 Berlin \\
 Germany}
\email{felsner@math.tu-berlin.de}

\author[J.~Hodor]{Jędrzej Hodor}
\email{jedrzej.hodor@gmail.com}

\author[P.~Micek]{Piotr Micek}
\email{piotr.micek@uj.edu.pl}

\thanks{M.\ Briański, J.\ Hodor, P.\ Micek are partially supported by a Polish National Science Center grant (BEETHOVEN; UMO-2018/31/G/ST1/03718).}

\maketitle


\begin{abstract}
    We study reconfiguration of independent sets in interval graphs under the token sliding rule.
    We show that if two independent sets of size $k$ are reconfigurable in an $n$-vertex interval graph, then there is a reconfiguration sequence of length $\Oh(k\cdot n^2)$.
    We also provide a construction in which the shortest reconfiguration sequence is of length $\Omega(k^2\cdot n)$.

    As a counterpart to these results, we also establish that \isr\ is PSPACE-hard on incomparability graphs, of which interval graphs are a special case.
\end{abstract}

\section{Introduction}

Let $G$ be a graph and let $k$ be a positive integer.
The \emph{reconfiguration graph} $R_k(G)$ of independent sets of size $k$ in $G$ 
has as its vertex set the set of all independent sets of size $k$ in $G$ and 
two independent sets $I$, $J$ of size $k$ in $G$ are adjacent in $R_k(G)$ 
whenever $I\triangle J=\set{u,v}$ and $uv \in E(G)$.

When two sets $I$ and $J$ are in the same component of $R_k(G)$
we say that $I$ and $J$ are \emph{reconfigurable}.
In such case, we can perform the following transformation process: 
(1) start by placing one token on each vertex of $I$; 
(2) in each step move one of the tokens to a neighboring vertex in $G$ 
but always keep the property that the vertices occupied by tokens induce an independent set in $G$;
(3) finish with tokens occupying all vertices of $J$.
Let $(I_0,\ldots,I_{m})$ be a path from $I$ to $J$ in $R_k(G)$ so $I_0=I$, $I_{m}=J$. 
For $i\in\set{1,\ldots,m}$, let $(u_i,v_i)$ be a pair of vertices in $G$ such that 
$I_{i-1}\setminus I_{i} = \set{u_i}$ and $I_{i}\setminus I_{i-1} = \set{v_i}$.
We say that the sequence $((u_1,v_1),\ldots,(u_{m},v_{m}))$ 
is a \emph{reconfiguration sequence} from $I$ to $J$ in $G$.

A large body of research is focused on the computational complexity of the corresponding decision problem --- \isr: 
given a graph $G$ and two independent sets $I$ and $J$, determine whether $I$ and $J$ are reconfigurable.
If we do not assume anything about the input graph this problem is PSPACE-complete, thus it is natural to investigate how its complexity changes when input graphs are restricted to a particular class of graphs.
Demaine et al.~\cite{DDF-EHIOOUY14} showed that the problem can be solved in polynomial time on trees.
Lokshtanov and Mouawad~\cite{LM18} proved that it remains PSPACE-complete on bipartite graphs.
Hearn and Demaine~\cite{HD05} showed that it is PSPACE-complete on planar graphs.
There are polynomial time algorithms for the class of cographs~\cite{KMM12} as well as claw-free graphs~\cite{BKW14}.
For a general survey of problems related to reconfiguration see~\cite{survey}.

Bonamy and Bousquet~\cite{BB17} presented an algorithm that given an interval graph $G$ and two independent sets of size $k$ in $G$ verifies if the independent sets are reconfigurable in $\Oh(n^3)$ time.
Curiously, their proof does not say anything meaningful about the length of the reconfiguration sequence if the two independent sets are reconfigurable.
Since there are at most $\binom{n}{k}$ independent sets of size $k$, we get a trivial bound on the length of the reconfiguration sequence, namely: $\Oh(n^k)$.
The question whether every two reconfigurable independent sets in an interval graph are connected by a reconfiguration sequence of polynomial length of degree independent of $k$ was raised by Bousquet~\cite{B20}, which we answer in the affirmative in \cref{thm:upper-bound}.

\subsection{Our results}

In this paper we will be mainly considering interval graphs. Our main results are stated in \cref{thm:upper-bound,thm:lower-bound} below.

\begin{theorem}
\label{thm:upper-bound}
Let $G$ be an $n$-vertex interval graph and $k$ be a positive integer. 
Then every component of $R_k(G)$ has diameter in $\Oh(k\cdot n^2)$. 

Moreover, there is a polynomial time algorithm that given $G$, $k$, and two independent sets $I$, $J$ of size $k$ in $G$ decides if $I$ and $J$ are reconfigurable and if so outputs a reconfiguration sequence connecting them of length $\Oh(k\cdot n^2)$.
\end{theorem}

A lower bound construction shows that the bound on the length of a reconfiguration sequence given in
\cref{thm:upper-bound} is close to tight.
\begin{theorem}
\label{thm:lower-bound}
For all integers $m\geq1$ and $k\geq1$, 
there is an interval graph $G_{m,k}$ with $|V(G_{m,k})|$ in $\Oh(m+k)$ and two reconfigurable independent sets $I$, $J$ of size $k$ in $G_{m,k}$ such that every reconfiguration sequence connecting $I$ and $J$ in $G$ is of length $\Omega(k^2\cdot m)$.
\end{theorem}

In light of \cref{thm:lower-bound}, so long as \(k \in \Theta(m)\), the bound of \cref{thm:upper-bound} is asymptotically tight. However, if \(k\) is small compared to $n$, the bounds of \cref{thm:upper-bound,thm:lower-bound} can differ by a factor of $O(n)$.
This leads to an interesting problem in its own right. The authors are not aware of any example giving a superlinear lower bound on the length of a reconfiguration sequence when the number of tokens is constant---\emph{even on the class of all graphs}. 
Specifically, the case $k=2$ remains open. 

Interval graphs are a special case in the more general class of incomparability graphs. As a somewhat expected (and indeed easy) result we have obtained the following theorem.

\begin{theorem}\label{thm:hardness}
    \isr\ is PSPACE-hard on incomparability graphs, even on incomparability graphs of posets of width at most \(w\) for some constant \(w \in \mathbb{N}\).
\end{theorem}

Another interesting specialization of the incomparability graphs are \emph{permutation graphs}. 
In~\cite{bipartitepermutation} the authors show a polynomial time algorithm solving \isr\ on bipartite permutation graphs.
We suspect that general permutation graphs should be amenable to methods similar to the tools we use here on interval graphs, however we have not been able to successfully apply them. 

\begin{conjecture}
    \isr\ is solvable in polynomial time when restricted to permutation graphs.
\end{conjecture}

\section{Preliminaries}
A \emph{component} of a graph $G$ is a non-empty induced subgraph of $G$ that is connected and is vertex-maximal under these properties.
The \emph{length} of a path is the number of edges in the path.
The distance of two vertices $u$, $v$ in a graph $G$ is the minimum length of a path connecting $u$ and $v$ in $G$.
The \emph{diameter} of a connected graph is the maximum distance between any pair of vertices in the graph.

A graph $G$ is an \emph{interval graph} if the vertices of $G$ can be associated with intervals on the real line in such a way that two vertices are adjacent in $G$ if and only if the corresponding intervals intersect.
Given an interval graph we will always fix an interval representation and identify the vertices with the corresponding intervals. 
We will also assume that all the endpoints of intervals in the representation are pairwise distinct. 
This can be easily achieved by perturbing the endpoints where it is needed.

Let $S$ be a reconfiguration sequence from $I$ to $J$ in $G$. 
We sometimes say that we \emph{apply}~$S$ to $I$ in $G$ and obtain an independent set $S(I)$.
When $S=((u,v))$ we also simply say that we \emph{apply} a pair $(u,v)$ to $I$ and obtain $S(I)=(I\setminus\set{u})\cup\set{v}$.

Given a sequence $S$, we denote by $\restr{S}{t}$ the prefix of $S$ of length $t$ (so $\restr{S}{0}$ is the empty sequence).
Given a reconfiguration sequence $S=((u_1,v_1),\ldots,(u_m,v_m))$ from $I$ to $J$ in $G$, 
clearly for each $t\in\set{0,\ldots,m}$, the prefix $\restr{S}{t}$ is a reconfiguration sequence from $I$ to the set $\restr{S}{t}(I)$ in $G$.
Thus, we also have $\restr{S}{0}(I)=I$.

\section{Upper bound: The Algorithm}

In this section we are going to prove Theorem~\ref{thm:upper-bound}.
For brevity, we omit implementation details as well as running time analyses of the algorithms outlined in this section. We note that a straightforward implementation of \cref{alg:solve} below (which is the procedure whose existence implies \cref{thm:upper-bound}) would yield a \(\mathcal{O}(n^3)\) algorithm.

Let $G$ be an $n$-vertex interval graph and let $k$ be a positive integer.
We fix an interval representation of $G$ distinguishing all the endpoints.
There are two natural linear orderings on the vertices of $G$: 
$\leq_{\text{left}}$ the order increasing along the left endpoints of the intervals and
$\leq_{\text{right}}$ the order increasing along the right endpoints of the intervals.

An independent set in $G$ is a set of pairwise disjoint intervals, as such they are naturally ordered on the line.
Thus given an independent set $A=\set{a_1,\ldots,a_{\ell}}$ in $G$ we will treat it as a tuple of intervals $(a_1,\ldots,a_{\ell})$ with $a_1 < \cdots < a_{\ell}$ on the line.
We define the projection $\pi_i(A)=a_i$ for each $i\in\set{1,\ldots,\ell}$.
Also when we apply $(u,v)$ to an independent set $A$, 
we say that $(u,v)$ \emph{moves the $i$-th token} of $A$ when $u=a_i$.

Let $\ell$ be a positive integer and let \(\calC\) be a non-empty family of independent sets each of size $\ell$ in  $G$. 
For $j\in\set{1,\ldots,\ell}$ we define
\[
    \ex_j (\calC, \text{left}) = \min_{\leq_{\text{right}}} \set{ \pi_j(A) \ \colon \ A \in \calC }, \quad
    \ex_j (\calC, \text{right}) = \max_{\leq_{\text{left}}} \set{ \pi_j(A) \ \colon \ A \in \calC }.
\]
For $p \in \set{0,\dots,\ell}$, we define the \emph{$p$-extreme set} of $\calC$ to be
\[
\bigcup_{1\leq j \leq p} \set{\ex_j(\calC,\mathrm{left})}\quad \cup 
\bigcup_{p+1\leq j \leq \ell} \set{\ex_j(\calC,\mathrm{right})}.
\]

We are going to show (\cref{lem:repsets_alt}) that every component $\calC$ of $R_k(G)$ contains all 
its $p$-extreme sets for $p\in\set{0,\ldots,k}$. 
This gives a foundation for our algorithm: given two independent sets $I$ and $J$ as the input, we are going to devise two reconfiguration sequences, transforming $I$ into the $(k - 1)$-extreme set of its component in $R_k(G)$ and $J$ into the $(k - 1)$-extreme set of its component in $R_k(G)$ respectively.
As the \((k - 1)\)-extreme set is a function of a connected component, we conclude that \(I\) and \(J\) are reconfigurable if and only if the obtained \((k - 1)\)-extreme sets are equal.
Thus, the essence of our work is to show that every independent set $I$ can be reconfigured in $\Oh(k\cdot n^2)$ steps into the $(k - 1)$-extreme set of its component in $R_k(G)$.

The technical lemma below is our basic tool used to reason about and manipulate reconfiguration sequences.

\begin{lemma}\label{lem:splicing}
  Let \(A = (a_1, \dots, a_\ell)\) and \(X = (x_1, \dots, x_\ell)\) be two independent sets in \(G\) and let \(S = ((u_1, v_1), \dots, (u_m, v_m))\) be a reconfiguration sequence from \(A\) to \(X\). With \(\calA = \set*{\restr{S}{t} (A) \ \colon \ t \in \set{0, \dots, m}}\)
  we denote the of all independent sets traversed from \(A\) to \(X\) along \(S\).
  Suppose that there are $i,j \in \set{0, \dots, \ell+1}$ with $i < j$ such that
    \begin{alignat*}{3}
        &\ex_{i}(\calA, \mathrm{left}) &= x_{i} &\quad\text{if } i\geq1, \text{ and}\\
        &\ex_{j} (\calA, \mathrm{right})\  &= x_{j} &\quad\text{if } j\leq\ell.
    \end{alignat*}

    Then \(A' = (x_1, x_2, \dots, x_{i}, a_{i+1}, \dots, a_{j-1} , x_{j}, \dots, x_\ell)\) is also an independent set in \(G\). 
    Moreover, if we let $S'$ be $S$ restricted to those pairs $(u_t,v_t)$ with $u_t$ being at a position $p$ with $i<p<j$
    in $\restr{S}{t-1}(A)$,
    then $S'$ is a reconfiguration sequence from \(A'\) to \(X\) in $G$.
\end{lemma}
\begin{remark}\label{rem:splicing}
    Since none of the first \(i\) tokens nor the last \(l + 1 - j\) tokens are affected by \(S'\) we can alternatively conclude that \(S'\) transforms \((a_{i + 1}, \dots, a_{j - 1})\) into \((x_{i + 1}, \dots, x_{j - 1})\) in \(G \backslash \bigcup \limits_{p \not \in \set{i + 1, \dots, j - 1}} N[x_p]\)
\end{remark}
\begin{proof}
    Let $B=(b_1,\ldots,b_{\ell})$ be a set in $\calA$. 
    The \emph{aligned} set of $B$ is defined as $(x_1,\ldots,x_{i},b_{i+1},\ldots,b_{j-1},x_{j},\ldots,x_{\ell})$.
    We claim that the aligned set of $B$ is an independent set in $G$.
    Note that some of the three parts $Y_1=(x_1,\ldots,x_{i})$, $Y_2=(b_{i+1},\ldots,b_{j-1})$, $Y_3=(x_{j},\ldots,x_{\ell})$ might be empty.
    Since all three parts are contained in an independent set, i.e.\ $X$ or $B$, 
    all we need to show is that 
    (1) if $Y_1$ and $Y_2$ are non-empty, then $x_{i}$ is completely to the left of $b_{i+1}$, and 
    (2) if $Y_2$ and $Y_3$ are non-empty, then $b_{j-1}$ is completely to the left of $x_{j}$.
    Thus,  suppose that $Y_1$ and $Y_2$ are non-empty, so $x_{i}$ and $b_{i+1}$ exist.
    By the assumptions of the lemma $x_{i} =\ex_{i}(\calA, \mathrm{left}) \leq_{\mathrm{right}} b_{i}$ and clearly $b_{i}$ is completely to the left of $b_{i+1}$. Therefore, $x_{i}$ is completely to the left of $b_{i+1}$ as desired. 
    Symmetrically, if $Y_2$ and $Y_3$ are non-empty, then
    $b_{j+1} \leq_{\mathrm{left}} x_{j+1}$ and $b_j$ is completely to the left of $b_{j+1}$. 
    Therefore, $b_{j}$ is completely to the left of $x_{j+1}$ as desired.  

    Since $A'$ is the aligned set of $A$, we conclude that $A'$ is independent in $G$.
    Let $m'$ be the length of $S'$ and $S'=((u'_1,v'_1),\ldots,(u'_{m'},v'_{m'}))$.
    Since $S'$ is a subsequence of $S$, we can fix $\phi(t)$, for each $t\in\set{1,\ldots,m'}$ such that the pair $(u_{\phi(t)},v_{\phi(t)})$ in $S$ corresponds to $(u'_{t},v'_{t})$ in $S'$.
    Let $A_{\phi(t)}=\restr{S}{\phi(t)}(A)$ and $A'_t = \restr{S'}{t}(A')$.
    By construction we have
    \begin{align*}
    \pi_p(A'_{t}) &= \pi_p(X)&&\textrm{if $p\in\set{1,\ldots,i}\cup\set{j,\ldots,\ell}$}\\
    \pi_p(A'_{t}) &= \pi_p(A_{\phi(t)})&&\textrm{if $p\in\set{i+1,\ldots,j-1}$}.
    \end{align*}
    Thus for each $t\in\set{0,\ldots,m'}$ we have that $\restr{S'}{t}(A')$ is the aligned set of $A_{\phi(t)}\in \calA$.
    Therefore $\restr{S'}{t}(A')$ is independent in $G$ which completes the proof that $S'$ is a reconfiguration sequence from $A'$ to $X$.
\end{proof}

\begin{lemma}\label{lem:repsets_alt}
Let $H$ be a non-empty induced subgraph of $G$, 
let $\ell\in\set{1,\ldots,k}$, and 
let $\calC$ be a component of $R_{\ell}(H)$. 
For every $p \in \set{0,\dots,\ell}$, 
the $p$-extreme set of $\calC$ is independent in $H$ and lies in $\calC$.
\end{lemma}
\begin{proof}
    Fix $p\in\set{0,\ldots,\ell}$. 
    Let $X = (x_1,\dots,x_\ell)$ be the $p$-extreme set of $\calC$.
    We claim that for every $i\in\set{0,\ldots,p}$ and $j\in\set{p+1,\ldots,\ell+1}$, 
    there is a set $A_{i,j}\in\calC$ such that 
    for all $q \in \set{1,\dots,i} \cup \set{j,\dots,\ell}$, we have $\pi_{q}(A_{i,j}) = x_q$.
    We prove this claim by induction on $m=i+(\ell+1-j)$. 
    For the base case, when $m=0$, so $i=0$ and $j=\ell+1$, we simply choose $A_{0,\ell+1}$ to be any element in $\calC$.


    For the inductive step, consider $m>0$, so $i>0$ or $j<\ell+1$. 
    The two cases are symmetric, so let us consider only the first one. 
    Thus suppose $i>0$ and therefore by the induction hypothesis, we get
    an independent set $A_{i-1,j}\in \calC$ of the form
    \[
    (x_1,\dots,x_{i-1},a_i, \dots, a_{j-1}, x_{j}, \dots, x_{\ell}),
    \]
    where $a_i,\ldots,a_{j-1}$ are some vertices of $H$.
    Let $B\in\calC$ such that $\pi_i(B)= \ex_i(\calC,\text{left}) = x_i$.
    Since $B, A_{i-1,j}\in \calC$, there is a reconfiguration sequence $S$ 
    from $B$ to $A_{i-1,j}$. 
    By \cref{lem:splicing}, we obtain that
    \[
    (x_1,\dots,x_{i-1},x_i,b_{i+1},\ldots,b_{j-1}, x_{j}, \dots, x_{\ell}) \in \calC.
    \]
    This set witnesses the inductive condition for $(i,j)$ and finishes the inductive step.
\end{proof}

When $k=1$, there is a single token in the graph which moves along a path in the interval graph.
This is a rather trivial setting still, we give explicit functions (\cref{alg:pushtoken}) to
have a good base for the general strategy.

    \begin{algorithm}
        \begin{algorithmic}[1]
            \Function{PushTokenLeft}{$H, u$}
                \State \(w := \min_{\leq_{\text{right}}} \set{v \in V(H) \ \colon \ v \text{ is reachable from } u \text{ in } H}\)
                \State \(\text{let} \ (v_0, \dots, v_m) \ \text{be the shortest path from} \ u \ \text{to} \ w \ \text{in} \ H \)
                \State \Return \([w, ((v_0,v_1), \dots, (v_{m-1}, v_m))]\)
            \EndFunction
            \Function{PushTokenRight}{$H, u$}
                \State \(w := \max_{\leq_{\text{left}}} \set{v \in V(H) \ \colon \ v \text{ is reachable from } u \text{ in } H}\)
                \State \(\text{let} \ (v_0, \dots, v_m) \ \text{be the shortest path from} \ u \ \text{to} \ w \ \text{in} \ H \)
                \State \Return \([w, ((v_0,v_m), \dots, (v_{m-1}, v_m))]\)
            \EndFunction
        \end{algorithmic}
        \caption{}
        \label{alg:pushtoken}
    \end{algorithm}
    \begin{proposition}\label{prop:k=1}
        Let $H$ be an interval graph, and $u$ be a vertex in $H$. 
        Then $\textsc{PushTokenLeft}(H,u)$ outputs the $\le_\mathrm{right}$-minimum vertex $w$ in the component of $u$ in~$H$ and a witnessing $u$-$w$ path $(v_0,v_1), \dots, (v_{m-1}, v_m)$ in $H$, where $v_0=u,v_m=w$ and $v_m <_\mathrm{right} \dots <_\mathrm{right} v_2 <_{\mathrm{right}} v_1$ and if $m\geq2$, $v_2<_{\mathrm{right}} v_0$. 

        Symmetrically, $\textsc{PushTokenRight}(H,u)$ outputs the $\le_\mathrm{left}$-maximum vertex $w$ in the component of $u$ in $H$ and a witnessing $u$-$w$ path $(v_0,v_1), \dots, (v_{m-1}, v_m)$ in $H$, where $v_0=u,v_m=w$ and $v_m >_\mathrm{left} \dots >_\mathrm{left} v_2 >_{\mathrm{left}} v_1$ and if $m\geq2$, $v_2>_{\mathrm{left}} v_0$. 

    \end{proposition}
    \begin{proof}
        We prove the first part of the statement about $\textsc{PushTokenLeft}(H,u)$.
        The second part is symmetric.
        Consider a shortest path $v_0v_1\cdots v_m$ from $u$ to $w$ in $H$. 
        Note that there is no point on the line that belongs to three intervals in the path as otherwise we could make the path shorter. 
        It is easy to see that 
        $v_{i+1}$ left overlaps $v_i$, i.e., $v_{i+1} <_{\mathrm{right}} v_i$ and $v_{i+1}<_{\mathrm{left}} v_i$ for all $i\in\set{0,\ldots, m-1}$ possibly except two cases: 
        $v_m$ may be contained in $v_{m-1}$ and (if $m\geq2$) $v_1$ may contain $v_0$.
        See Figure~\ref{fig:shortest-path} for an illustration of such a shortest path in an interval graph.
    \end{proof}
\begin{figure}[!h]
    \centering
    \includegraphics{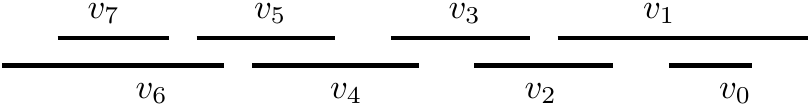}
    \caption{A shortest path from $v_0$ to $v_7$ where $v_7$ is an interval with leftmost right endpoint.}
    \label{fig:shortest-path}
\end{figure}

    For $k=2$, we give an algorithm, see~\cref{alg:pushapart},
     that finds a short reconfiguration from a given independent set $A$ in $H$ to 
     the $1$-extreme set of the component of $A$ in a reconfiguration graph $R_2(H)$.

    \begin{algorithm}
        \begin{algorithmic}[1]
            \Function{PushApart}{$H, A$}
                \State \(a := \pi_1(A) \ b := \pi_2(A)\)
                \State \(S := ()\)
                \Do \label{alg:pushapart:beginloop}
                    \State \([a, S_1]:= \textsc{PushTokenLeft}(H \backslash N[b], \ a)\) \label{alg:pushapart:pushleft}
                    \State \([b, S_2]:= \textsc{PushTokenRight}(H \backslash N[a], \ b)\) \label{alg:pushapart:pushright}
                    \State \(S := \textsc{Concat}(S,S_1,S_2)\)
                \doWhile {\(\textsc{Concat}(S_1,S_2) \neq \emptyset\)}
                \State \Return \((a, b), S\)
            \EndFunction
        \end{algorithmic}
        \caption{}
        \label{alg:pushapart}
    \end{algorithm}

    \begin{proposition}\label{prop:k=2}
        Let $H$ be an interval graph and $A$ be an independent set of size $2$ in $H$.
        Then $\textsc{PushApart}(H,A)$ outputs the $1$-extreme set of 
        the component of $A$ in $R_2(H)$ and a reconfiguration sequence from $A$ to the $1$-extreme set of length at most $2|V(H)|$.
    \end{proposition}
    \begin{proof}
        Suppose that $\textsc{PushApart}(H,A)$ outputs $(a^{\star},b^{\star})$ and let 
        $(x_1,x_2)$ be the $1$-extreme set 
        of the component $\calC$ of $A$ in $R_2(H)$.
        Suppose to the contrary that $(a^{\star},b^{\star}) \neq (x_1,x_2)$
        Thus, 
        either $x_1 = \ex_1(\calC,\mathrm{left}) <_{\mathrm{right}} a^{\star}$ or
        $b^{\star} <_{\mathrm{left}} \ex_2(\calC,\mathrm{right}) =x_2$.

        Consider a path $((a_0,b_0),\ldots,(a_m,b_m))$ from $(a^{\star},b^{\star})$ to $(x_1,x_2)$ in $R_2(G)$.
        Let $i$ be the smallest index such that $a_i <_{\mathrm{right}} a_0 = a^{\star}$ or $b^{\star} <_{\mathrm{left}} b_i$. 
        This index is well-defined as $i=m$ satisfies the condition and obviously $i>0$.
        Now suppose that $a_i <_{\mathrm{right}} a^{\star}$. 
        The proof of the other case goes symmetrically.
        By the minimality of $i$ we have 
        \[
        a_i <_{\mathrm{right}} a^{\star} \leq_{\mathrm{right}} a_{i-1}\qquad\text{and}\qquad
        b_i \leq_{\mathrm{left}} b^{\star}.
        \]
        The inequalities on the right endpoints of $a_i$, $a^{\star}$, and $a_{i-1}$ imply that these three intervals form a connected subgraph of $H$. 
        Since $a_{i-1}$ and $a_i$ are in the same component of $H\setminus N[b_i]$, 
        we conclude that $a^{\star}$ is also in that component.
        Finally, since $b_i \leq_{\mathrm{left}} b^{\star}$ 
        we conclude that $a_{i-1}$, $a^{\star}$, $a_i$ are together in the same component of $H \setminus N[b^{\star}]$. 
        See Figure~\ref{fig:PushApart} that illustrates all these inequalities.
\begin{figure}[!h]
    \centering
    \includegraphics{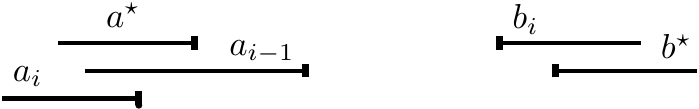}
    \caption{$a_i <_{\mathrm{right}} a^{\star} \leq_{\mathrm{right}} a_{i-1}$ and 
    $b_i \leq_{\mathrm{left}} b^{\star}$. The intervals $a_i$, $a^{\star}$, $a_{i-1}$ are together in a component of $H\setminus N[b_i]$ and so they are together in a component of $H\setminus N[b^{\star}]$ as well.}
    \label{fig:PushApart}
\end{figure}
        Consider the last iteration of the loop in $\textsc{PushApart}(H,A)$. 
        The variables $a$ and $b$ keep the values $a^{\star}$ and $b^{\star}$ in this iteration. 
        In particular $\textsc{PushTokenLeft}(H\setminus N[b^{\star}],a^{\star})$ did not change the value of $a$. 
        This is a contradiction as by~\cref{prop:k=1} $\textsc{PushTokenLeft}(H\setminus N[b^{\star}],a^{\star})$ outputs the $\leq_{\mathrm{right}}$ minimum vertex in the component of $a^{\star}$ in $H\setminus N[b^{\star}]$ but we already know that $a^{\star}$ is not $\leq_{\mathrm{right}}$-minimum there.

        Now, let us prove the claim about the length of the output reconfiguration sequence $S$. Let $(a_0,b_0), \dots, (a_m,b_m)$ be the path in $R_2(H)$ corresponding to $S$, 
        that is $(a_t,b_t) = S|_t(A)$ for $t\in\set{0,\dots,m}$. 
        Suppose that $(a,b) = (a^0,b^0)$ at the beginning of the loop in line \ref{alg:pushapart:beginloop} and $(a,b) = (a^1,b^1)$ after pushing left and right in lines \ref{alg:pushapart:pushleft}-\ref{alg:pushapart:pushright}. 
        By~\cref{prop:k=1}, we know that $a^1 \le_{\text{right}} a^0$ and $b^0 \le_{\text{left}} b^1$ and every intermediate configuration $(a',b')$ on the path satisfies $a^1 \le_{\text{right}} a' \le_{\text{right}} a^0$ and $b^0 \le_{\text{left}} b' \le_{\text{left}} b^1$ with a possible exception of the first move from $a^0$ and the first move from $b^0$.
        This way we see that the total number of steps without the exceptional moves is at most $|V(H)|$ and each interval in $H$ can be the target of at most one exceptional move. This gives a bound $2|V(H)|$, as desired.
    \end{proof}

    We present now our main reconfiguration algorithm, see~\cref{alg:solve}.

        \begin{algorithm}
            \begin{algorithmic}[1]
                \Function{Reconfigure}{$G, k, A$}
                    \For { \(i \in \set{1,2,\dots,k}\) }
                        \State \(a_i := \pi_i[A]\)
                        \State \(\text{lext}_i := a_i\) \Comment{$\text{lext}_k$ is unused}\label{alg:solve:init1}
                        \State \(\text{rext}_i := a_i\) \Comment{$\text{rext}_1$ is unused}\label{alg:solve:init2}
                    \EndFor
                    \State \(j := 1;\ S=()\)
                    \While{ \(j < k\) } \label{alg:solve:outerloop}
                        \State \([(a_j,a_{j+1}),S'] := \) \Call{PushApart}{$G \setminus \bigcup \limits_{i \neq j,j+1} N[a_i], \set{a_j, a_{j + 1}}$}\label{alg:solve:PushApart}
                        \vspace{-1.0ex}
                        \State \(S.\textsc{Append}(S')\)
                        \State \(\gamma=1\)
                        \If { \((a_j, a_{j+1}) \neq (\text{lext}_j, \text{rext}_{j + 1})\) }
                            \State \(\text{lext}_j := a_j\) \label{alg:solve:update-lext}
                            \State \(\text{rext}_{j + 1} := a_{j+1}\) \label{alg:solve:update-rext}
                            \IIf {$j>1$} \(\gamma=-1\)\EndIIf
                        \EndIf
                        \State \(j:=j+\gamma\)\label{alg:solve:j-update}
                    \EndWhile
                    \State \Return $(a_1,\ldots,a_k),\, S$
                \EndFunction
            \end{algorithmic}
            \caption{}
            \label{alg:solve}
        \end{algorithm}

    \begin{lemma}
    Let $k\geq2$ and let $A$ be an independent set of size $k$ in $G$.
    Then $\textsc{Reconfigure}(G, k, A)$ outputs the $(k-1)$-extreme set of the component 
    of $A$ in $R_{k}(G)$ and a reconfiguration sequence from $A$ to this set of length 
    $\Oh(k\cdot n^2)$.
    \end{lemma}

    \begin{proof}

        We begin our consideration of \cref{alg:solve} by noting two invariants.

        \begin{claim}\label{inv:syntactic}
            Every time~\cref{alg:solve} reaches line~\ref{alg:solve:j-update}, 
            we have
            \[
            (a_1,\ldots,a_k) = (\text{lext}_1,\ldots,\text{lext}_j,\text{rext}_{j+1},\ldots,\text{rext}_k).
            \]
        \end{claim}
        \begin{proof}
            Note that the equation holds after the initialization in lines~\ref{alg:solve:init1}-\ref{alg:solve:init2}. 
            Later on, the values of $(a_1,\ldots,a_k)$ are updated only in line~\ref{alg:solve:PushApart} and if so then corresponding values of $\text{lext}$ and $\text{rext}$ are updated in lines~\ref{alg:solve:update-lext}-\ref{alg:solve:update-rext}.
        \end{proof}

        \begin{claim}\label{inv:gapsexpand}
            For every $j\in\set{2,\ldots,k}$, 
            the values held by $\text{lext}_j$ are nonincreasing with respect to \(\leq_{\mathrm{right}}\).
            Symmetrically, for every $j\in\set{1,\ldots,k-1}$, 
            the values held by $\text{rext}_j$ are nondecreasing with respect to \(\leq_{\mathrm{left}}\).
        \end{claim}
        \begin{proof}
            We prove the statement by induction over the iterations steps in~\cref{alg:solve}.
            Consider the moment when the values of $\text{lext}_j$ and $\text{rext}_{j+1}$ are updated in lines~\ref{alg:solve:update-lext}-\ref{alg:solve:update-rext}.
            Let \((\ell_1, \ell_2, \dots, \ell_k)\) and \((r_1, r_2, \dots, r_k)\) be the values held by vectors \(\text{lext}\) and \(\text{rext}\) just after this update.
            Let \((\ell'_1, \dots, \ell'_k)\) and \((r'_1, \dots, r'_k)\) be the values held by vectors \(\text{lext}\) and \(\text{rext}\) just after the previous update of \(\text{lext}_j\) or \(\text{rext}_{j + 1}\) in lines~\ref{alg:solve:update-lext}-\ref{alg:solve:update-rext} or the initial values of these vectors (assigned in lines~\ref{alg:solve:init1}-\ref{alg:solve:init2}) if there was no previous update.
            All we need to show is that \(\ell_j \leq_{\text{right}} \ell'_j\) and \(r'_{j + 1} \leq_{\text{left}} r_{j + 1}\).

            By~\cref{inv:syntactic} we know that the algorithm had tokens in \((\ell'_1, \dots, \ell'_j, r'_{j + 1} \dots, r'_k)\) and after applying some reconfiguration sequence, say $S=((u_1,v_1),\ldots,(u_m,v_m))$, it reached the configuration \((\ell_1, \dots, \ell_j, r_{j + 1}, \dots, r_k)\). 
            Let $\calA=\set{\restr{S}{t}(A):\ t\in\set{0,\ldots,m}}$. 

            By induction hypothesis the statement holds for all the updates applied so far by~\cref{alg:solve}. In particular, $\ell_{j-1}$ is the $\leq_{\text{right}}$-minimal position of the $(j-1)$-th token so far and $r_{j+2}$ it the $\leq_{\text{left}}$-maximal position of the $(j+2)$-th token so far (assuming that these tokens exist). Thus,
            \begin{align*}
            \ell_{j-1} &= \ex_{j-1}(\calA, \text{left})&&\text{if $j-1\geq1$,}\\
            r_{j+2} &= \ex_{j+2}(\calA, \text{right})&&\text{if $j+2\leq k$.}
            \end{align*}
            Therefore, we may apply~\cref{lem:splicing,rem:splicing} and conclude that 
            $(\ell'_j,r'_{j+1})$ and \((\ell_j, r_{j + 1})\) are in the same component 
            of $R_2(G \setminus \bigcup \limits_{i \neq j,j+1} N[a_i])$.
            By~\cref{prop:k=2}, the execution of 
            $\textsc{PushApart}$ in line~\ref{alg:solve:PushApart} outputs the $1$-extreme set, namely $(\ell_j,r_{j+1})$, of the component of $(\ell'_j,r'_{j+1})$ in $R_2(G \setminus \bigcup \limits_{i \neq j,j+1} N[a_i])$. Therefore, $\ell_j \leq_{\mathrm{right}}\ell'_j$ and $r_{j+1} \geq_{\mathrm{left}} r'_{j+1}$, as desired.
        \end{proof}

        \begin{claim}\label{cl:correctness}
            Consider a moment when~\cref{alg:solve} reaches the line~\ref{alg:solve:j-update} and \(\gamma = 1\). Let $\alpha$ be the value of variable $j$ prior to the update.
            Let $(\ell_1,\ldots,\ell_k)$, $(r_1,\ldots,r_k)$ be the values held at this moment by vectors $\text{lext}$ and $\text{rext}$, respectively.
            Then,
            \begin{align*}
                \ell_i &= \ex_i \paren*{\calC, \text{left} }&&\text{for $i\in\set{1,\ldots,\alpha}$},\\
                 r_{\alpha + 1} &= \ex_{\alpha + 1 } \paren*{\calC, \text{right}},
            \end{align*}
            where \(\calC\) is the component of \((\ell_1, \ldots, \ell_\alpha, r_{\alpha + 1})\) in \(R_{\alpha + 1}(G \setminus \bigcup\limits_{i > \alpha + 1} N[r_i])\).
            In particular, $(\ell_1,\ldots,\ell_{\alpha},r_{\alpha+1})$ is the $\alpha$-extreme set in $\calC$.
        \end{claim}
        \begin{proof}
            We proceed by induction on \(\alpha\). 
            First we deal with \(\alpha = 1\). 
            When~\cref{alg:solve} starts an iteration of the while loop with $j=1$, 
            then by~\cref{inv:syntactic} (and initialization in lines~\ref{alg:solve:init1}-\ref{alg:solve:init2}), 
            we have $(a_3,\ldots,a_k)=(r_3,\ldots,r_k)$. 
            By~\cref{prop:k=2}, $\textsc{PushApart}(G\setminus\bigcup\limits_{i > 2} N[r_i],\set{a_1,a_2})$ 
            executed in line~\ref{alg:solve:PushApart} outputs the $1$-extreme set of the component of $\set{a_1,a_2}$ in $R_2(G\setminus\bigcup\limits_{i > 2} N[r_i])$.
            After the update in lines~\ref{alg:solve:update-lext}-\ref{alg:solve:update-rext}, this set is stored in $\set{\text{lext}_1,\text{rext}_2}=\set{\ell_1,r_2}$ when~\cref{alg:solve} reaches the line~\ref{alg:solve:j-update}, as desired.
            
            Let us assume that \(\alpha > 1\) and that the claim holds for all smaller values of $\alpha$.
            Consider an iteration of the while loop with $j=\alpha$ such that~\cref{alg:solve} reaches line~\ref{alg:solve:j-update} with $\gamma=1$.
            Let $(\ell_1,\ldots,\ell_k)$, $(r_1,\ldots,r_k)$ be the values held at this moment by vectors $\text{lext}$ and $\text{rext}$, respectively.
            For convenience, we call this iteration the \emph{present} iteration.
            
            Now starting from the present iteration consider the last iteration before with $j=\alpha-1$.
            We call this iteration the \emph{past} iteration.
            Clearly, the past iteration had to conclude with $\gamma=1$ and all iterations between the past and the present (there could be none) must have the value of variable $j\geq\alpha$ and those with $j=\alpha$ must conclude with $\gamma=1$.
            This implies that the values of \((\text{lext}_1, \dots, \text{lext}_{\alpha})\) and \((\text{rext}_2, \dots, \text{rext}_{\alpha + 1})\) did not change between the past and the present iterations so they constantly are $(\ell_1,\ldots,\ell_{\alpha})$ and $(r_2,\ldots,r_{\alpha+1})$.

            Let \(\calD\) be the connected component of \((\ell_1, \ell_2, \dots, \ell_{\alpha - 1}, r_\alpha)\) in \(R_\alpha (G \backslash N[r_{\alpha + 1}])\). The inductive assumption for the past iteration yields:
            \begin{align}
                \ell_i &= \ex_i \paren*{\calD, \text{left}} &&\text{for $i\in\set{1,\ldots,\alpha-1}$, and}\tag{$\star$}\label{eq:ind-ell}\\
                r_{\alpha} &= \ex_{\alpha} \paren*{\calD, \text{right}},\notag
            \end{align}

            Note that in the iteration immediately following the previous iteration (this was an iteration with $\alpha=j$ and may be the present iteration) the only token movement was the travel of the $j$-th token from $r_{\alpha}$ to $\ell_{\alpha}$ (the $(j+1)$-th token stays at $r_{\alpha+1}$).
            In particular, there is a path from $r_{\alpha}$ to $\ell_{\alpha}$ in \(G \backslash (N[\ell_{\alpha - 1}] \cup N[r_{\alpha + 1}])\). 
            Therefore, there is a path connecting \((\ell_1, \dots, \ell_{\alpha - 1}, r_\alpha)\) and \((\ell_1, \dots, \ell_\alpha)\) in $R_\alpha (G \backslash N[r_{\alpha + 1}])$, so both independent sets are in \(\calD\).
         
            Now we argue, that
            \[
            \ell_\alpha = \ex_\alpha \paren*{ \calD, \text{left} }.
            \]
            Indeed, take any \( (v_1, v_2, \dots, v_\alpha) \in \calD\) and we aim to show that \(\ell_\alpha \leq_{\text{right}} v_\alpha\). 
            As \((v_1, \dots, v_\alpha)\) and \((\ell_1, \dots, \ell_\alpha)\) are in one component of the reconfiguration graph $R_\alpha (G \backslash N[r_{\alpha + 1}])$ and by~\eqref{eq:ind-ell}, we may apply \cref{lem:splicing} to conclude that \((\ell_1, \ell_2, \dots, \ell_{\alpha - 1}, v_\alpha)\) also lies in \(\calD\). 
            Moreover by \cref{rem:splicing}, the vertices \(\ell_\alpha\) and \(v_\alpha\) are connected by a path in \(G \backslash \paren*{ N[\ell_{\alpha - 1}] \cup N[r_{\alpha + 1}] }\).
            Thus by \cref{prop:k=2}, $\textsc{PushApart}$ executed in the present iteration guarantees that \(\ell_\alpha \leq_{\text{right}} v_\alpha\) as claimed.

            Let $\calC$ be the component of \((\ell_1, \ldots, \ell_\alpha, r_{\alpha + 1})\) in \(R_{\alpha + 1}(G \setminus \bigcup\limits_{i > \alpha + 1} N[r_i])\).
            We proceed to argue that
            \[
            r_{\alpha + 1} = \ex_{\alpha + 1} \paren*{\calC, \text{right}}.
            \]
            Assume for the sake of contradiction that there is some \(A \in \calC\), such that \( r_{\alpha + 1} <_{\text{left}} \pi_{\alpha + 1}(A) \) and 
            fix such an $A$ with a shortest possible reconfiguration sequence 
            \(S = ((u_1, v_1), \dots, (u_m, v_m))\) from \((\ell_1, \dots, \ell_\alpha, r_{\alpha + 1})\) to \(A\) in $R_{\alpha + 1}(G \setminus \bigcup\limits_{i > \alpha + 1} N[r_i])$.
            Let $A_t=\restr{S}{t}((\ell_1, \dots, \ell_\alpha, r_{\alpha + 1}))$, for all $t\in\set{0,\ldots,m}$.
            By the choice of $A$ and $S$, for every $t\in\set{0,\ldots,m-1}$ we have \(\pi_{\alpha + 1}(A_t) \leq_{\text{left}} r_{\alpha + 1}\). 
            We apply now \cref{lem:splicing} (with $i=0$, $j=\alpha+1$) to a path from \(A_0=(\ell_1, \dots, \ell_\alpha, r_{\alpha + 1})\) to \(A_{m - 1}\)
            and conclude that for each \(t\in\set{0,\ldots,m-1}\) the set $A'_t =(\pi_1(A_t),\ldots,\pi_{\alpha}(A_t),r_{\alpha+1})$ is an independent set in $\calC$.
            Consider now a path of independent sets of size $\alpha$ formed by dropping the $(\alpha+1)$-th coordinate of each set in the path  $(A'_0,\ldots,A'_{m-1},A_m)$.
            Since $(\pi_1(A_0),\ldots,\pi_{\alpha}(A_0))=(\ell_1,\ldots,\ell_{\alpha})\in\calD$, the  whole path lives in \(\calD\).
            Therefore, by~\eqref{eq:ind-ell} we have
            \[
            \ell_i \leq_\text{right} \pi_i (A_t),
            \]
            for all \(t\in\set{0,\ldots,m}\) and \(i\in\set{1,\ldots,\alpha}\).
            But this in turn allows us to apply \cref{lem:splicing,rem:splicing} once more (this time with $i=\alpha$, $j=\alpha+2$) to a path from  \((\ell_1, \dots, \ell_\alpha, r_{\alpha + 1})\) to \(A_{m}\) and we conclude that 
            \(r_{\alpha + 1}\) and \(\pi_{\alpha + 1}(A_{m})\) are in the same component of 
            \(G \backslash \paren*{N[\ell_\alpha] \cup \bigcup\limits_{i > \alpha + 1} N[r_i])}\). 
            But $\textsc{PushApart}$ executed in the present iteration outputs $(\ell_{\alpha},r_{\alpha+1})$ while \(r_{\alpha + 1} <_\text{left} \pi_{\alpha + 1} (A_{m})\).
            This contradicts \cref{prop:k=2} and completes the proof that 
            $r_{\alpha + 1} = \ex_{\alpha + 1} \paren*{\calC, \text{right}}$.

            It remains to prove that $\ell_i=\ex_i\paren*{\calC, \text{left}}$ for all $i\in\set{1,\ldots,\alpha}$.
            Pick an arbitrary \(A=(v_1, v_2, \dots, v_{\alpha + 1}) \in \calC\).
            Since we already know that 
            \(r_{\alpha + 1} = \ex_{\alpha + 1}(\calC, \text{right})\), 
            we can apply \cref{lem:splicing} (with $i=0$, $j=\alpha+1$) to a path from \((\ell_1, \dots, \ell_\alpha, r_{\alpha + 1})\) to \((v_1, v_2, \dots, v_{\alpha + 1})\), and we conclude that there is a reconfiguration sequence transforming \((v_1, \dots, v_\alpha)\) into \((\ell_1, \dots, \ell_\alpha)\) in \(G \setminus N[r_{\alpha+1}])\). 
            Thus, this path lies in $\calD$ and the desired inequalities $\ell_i \leq_{\text{right}} v_i$ for all $i\in\set{1,\ldots,\alpha}$ follow by~\eqref{eq:ind-ell}.
        \end{proof}

        Clearly, Claims~\ref{cl:correctness} and~\ref{inv:syntactic} establish the correctness of the algorithm. 
        Equipped with the invariant given by \cref{inv:gapsexpand}, we can bound the length of the returned reconfiguration sequence. Indeed, observe that in each iteration of the while loop in line~\ref{alg:solve:outerloop} either \(\text{lext}_j\) decreases wrt. \(\leq_\text{right}\), or \(\text{rext}_{j + 1}\) increases wrt. \(\leq_\text{left}\) while \(j\) drops by \(1\), or \(j\) increases by \(1\).
        Now this implies that the outer loops can iterate at most \(4nk + k\), as the quantity
        \[
            j + 2 \sum \limits_{i = 1}^{k} \text{Index}_{\leq_{\text{left}}}(\text{rext}_i) + (n - \text{Index}_{\leq_{\text{right}}}(\text{lext}_i) + 1)\text{,}
        \]
        where \(\text{Index}_{\leq}(x)\) denotes the position of element \(x\) in a given linear order \(\leq\) on some fixed finite set, increases by at least one in each iteration and it is at most \(4nk + k\).

        As seen in \cref{prop:k=2}, each call of the procedure \(\textsc{PushApart}\) returns a sequence consisting of at most \(2n\) moves. Therefore, the length of reconfiguration sequence returned by \cref{alg:solve} is at most \(8kn^2 + 2kn \in \mathcal{O}(kn^2)\). 
        This completes the proof of~\cref{thm:upper-bound}.
    \end{proof}

\section{Lower bound: Example}
We present a family of graphs $\set{G_{m,k}}_{m,k\ge 1}$, such that $|V(G_{m,k})| = 8k+2m-5$ and $R_k(G_{m,k})$ contains a component of diameter at least $\frac{k^2}{4} \cdot m$. This will prove \cref{thm:lower-bound}.

Fix integers $m,k \ge 1$. We will describe a family of intervals $\calI_{m,k}$. The graph $G_{m,k}$ will be simply the intersection graph of $\calI_{m,k}$. We construct the family in three steps. We initialize $\mathcal{I}_{m,k}$ with $(k-1) + (m+2k-1) + k$ pairwise disjoint intervals:
$$a_{k-1}, \dots, a_{1}, v_1, \dots, v_{m+2k-1}, b_1, \dots, b_k,$$
listed with their natural left to right order on the line. We call these intervals, the \emph{base} intervals. 
Let $N = m + 2k - 1$. We put into $\calI_{m,k}$ further $N-1$ intervals:
$$v_{1,2}, v_{2,3}, \dots, v_{N-1,N},$$
where for each $i \in \set{1,\dots,N-1}$, the interval $v_{i,i+1}$ is an open interval with the left endpoint in the middle of $v_i$ and the right endpoint in the middle of $v_{i+1}$. We call these intervals the \emph{path} intervals. Finally, we put into $\calI_{m,k}$ two groups of \emph{long} intervals:
$$\ell_1,\dots,\ell_{k-1} \ \text{and} \ r_1,\dots,r_k,$$
where for each $i\in \set{1,\dots,k-1}$ the interval $\ell_i$ is the open interval with the left endpoint coinciding with the left endpoint of $a_i$ and the right endpoint coinciding with the right endpoint of $v_{N-(k-1)-i}$. Symmetrically, for each $i\in \set{1,\dots,k}$ the interval $r_i$ is the open interval with the left endpoint coinciding with the left endpoint of $v_{k-i+1}$ and the right endpoint coinciding with the right endpoint of $b_i$. This completes the construction of $\calI_{m,k}$. See Figure~\ref{fig:G(6,3)-letters}.

\begin{figure}[!h]
    \centering
    \includegraphics{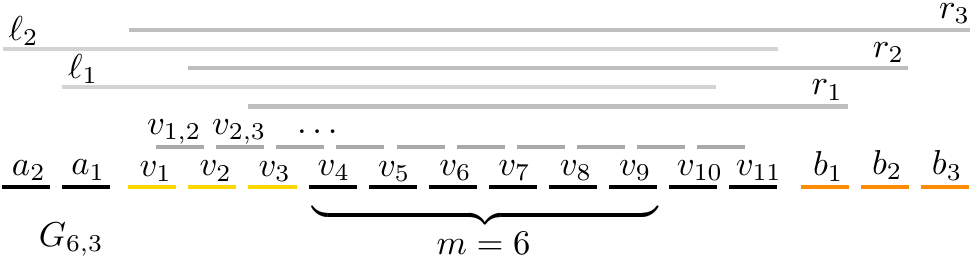}
    \caption{The graph $G_{6,3}$ with two distinguished independent sets $I = \set{v_1,v_2,v_3}$ and $J = \set{b_1,b_2,b_3}$.}
    \label{fig:G(6,3)-letters}
\end{figure}

Consider two independent sets $ I = (v_1,\dots,v_k) \ \text{and} \ J = (b_1,\dots,b_k)$ in $G_{m,k}$.

\begin{lemma}\label{lem:example}
The sets $I$ and $J$ are in the same component of $R_k(G_{m,k})$ and every reconfiguration sequence from $I$ to $J$ has length at least $\frac{k^2}{4}\cdot m$.
\end{lemma}
\begin{proof}
    We put most of the effort to prove the second part of the statement, that every reconfiguration sequence from $I$ to $J$ has length at least $\frac{k^2}{4}\cdot m$.

    We define a sequence of independent sets (see Figure~\ref{fig:example-reconfiguration}):
    \begin{align*}
        C_0 &= (v_1, \dots, v_k) = I, \\
        C_1 &= ( v_1, \dots, v_{k-1}, r_1 ), \\
        C_2 &= ( \ell_1, v_{N-(k-1)}, \dots, v_N, b_1 ), \\
        & \ \ \vdots \\
        C_{2i-1} &= ( a_{i-1}, \dots, a_{1}, v_{1}, \dots, v_{k-1}, r_{i} ), \\
        C_{2i} &= ( \ell_i, v_{N-(k-i-2)}, \dots, v_{N}, b_1, \dots, b_i ), \\
        & \ \ \vdots \\
        C_{2k} &= (b_1,\dots,b_k) = J.
    \end{align*}

    \begin{figure}[!h]
        \centering
        \includegraphics{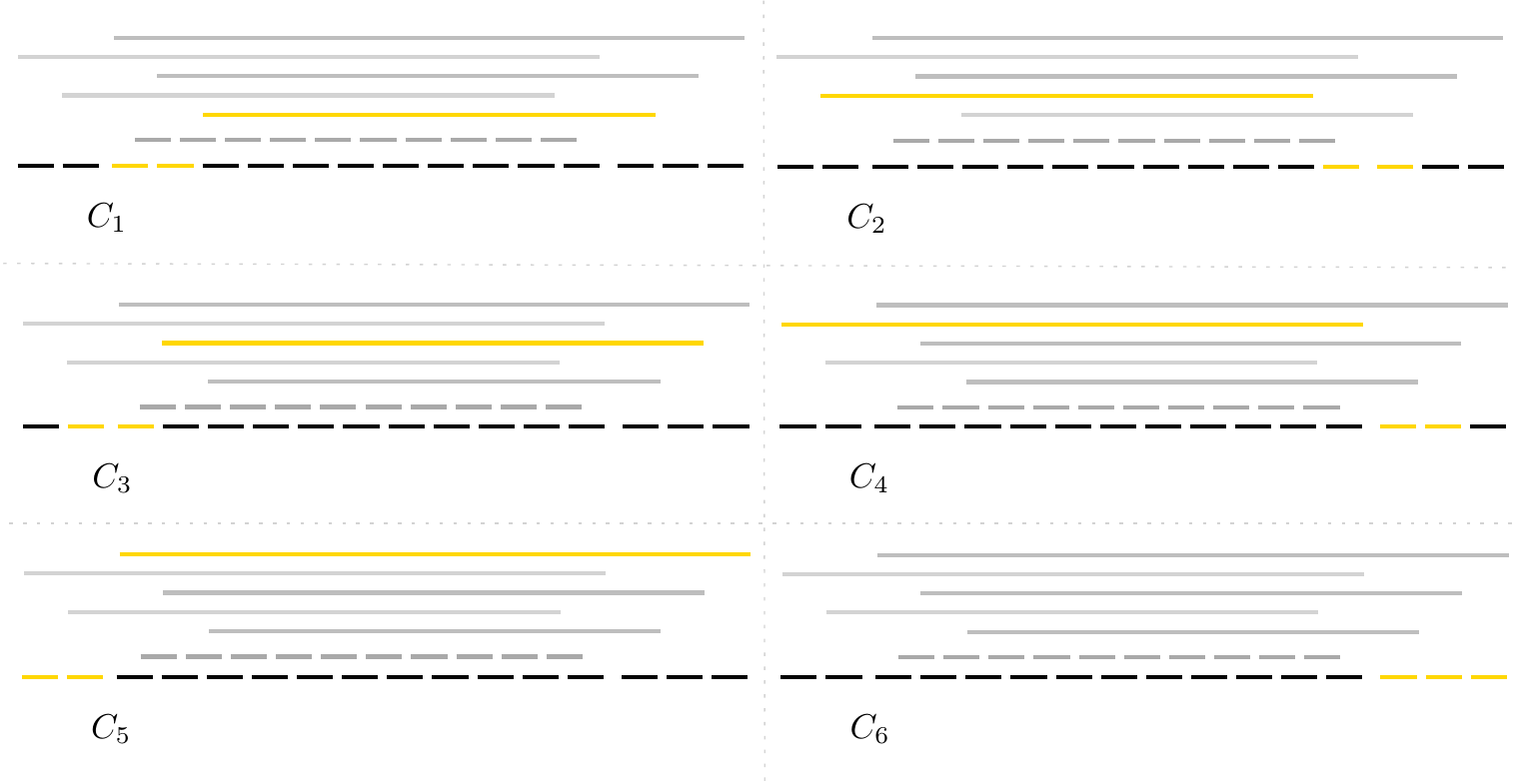}
        \caption{The sets $C_1,\dots,C_6$ in $G_{6,3}$.}
        \label{fig:example-reconfiguration}
    \end{figure}

    It is easy to construct a path from $C_j$ to $C_{j+1}$ in $R_k(G_{m,k})$ for $j\in \set{0,\dots,2k-1}$ which proves that $I,J$ are in the same component of $R_k(G_{m,k})$.

    Let $(K_0,\dots,K_M)$ be a path in $R_k(G_{m,k})$ from $I$ to $J$. The proof will follow from two claims. The first one is that $(C_0,\dots,C_{2k})$ is a subsequence of $(K_0,\dots,K_M)$, and the second one is that for every $i\in\set{1,\dots,k-1}$ every path from $C_{2i-1}$ to $C_{2i}$ is of length at least $(k-2i-1) \cdot m$. A symmetric argument can be used to bound the distance between $C_{2i}$ and $C_{2i+1}$ which we omit here as it would only improve the final lower bound by a constant factor.

    Let $P$ be the set of path intervals in $G_{m,k}$. Define for each $i \in\set{1,\dots,k-1}$ the following graphs:
    \begin{align*}
        H_{2i-1} &= G_{m,k}[\{ a_{i-1}, \dots, a_1, v_1, \dots, v_N, b_1, \dots, b_i \} \cup P], \\
        H_{2i} &= G_{m,k}[\{ a_{i}, \dots, a_1, v_1, \dots, v_N, b_1, \dots, b_i \} \cup P].
    \end{align*}

    Note that $C_0=K_0$ and $C_{2k} = K_M$, so $C_0$ and $C_{2k}$ occurs in $(K_0,\dots,K_M)$. Fix $j\in \set{0,\dots,2k-2}$. Suppose that the independent set $C_j$ occurs in $(K_0,\dots,K_M)$ and fix such an occurrence. We will argue that $C_{j+1}$ must occur afterwards in the sequence.

    Observe that all base intervals from $C_j$ are in $H_j$. However $b_k \notin H_j$ and $b_k \in K_M$, hence to reconfigure from $C_j$ to $K_M$ eventually a token has to be moved to some base interval not in $H_j$. Thus, let $X_j$ be the set of base intervals not in $H_j$, i.e.
    \[
        X_j =\begin{cases}\set{a_{k-1}\dots,a_{i}} \cup \set{b_{i+1},\dots,b_{k}} &\text{if} \ j \ \text{is odd,} \\
    \set{a_{k-1}\dots,a_{i+1}} \cup \set{b_{i+1},\dots,b_{k}} &\text{if} \ j \ \text{is even.}
    \end{cases}
    \]

    Note that the only neighbours of intervals in $X_j$ are long. Let $Y$ be the first independent set in $(K_0,\dots,K_M)$ that occurs after the fixed occurence of $C_j$ and contains a long interval $u_0$ neighbouring some element in $X_j$. We claim that $Y=C_{j+1}$.

    First, we show that $u_0 = \ell_{i}$ if $j = 2i-1$ and $u_0 = r_{i+1}$ if $j = 2i$ respectively. Assume for now that $j=2i-1$. Observe that for all $p\in\set{1,\dots,i-1}$ we have $N(\ell_p)\cap X_j =\emptyset$ and consequently $u_0\neq \ell_p$. On the other hand, for all $p\in\set{i+1,\dots,k-1}$ we have $\alpha \left( H_j \backslash N(\ell_p) \right)< k-1$. Therefore, whenever $u_0=\ell_p$ there is a token in $Y$ that is not in $H_j$. This contradicts the minimality of~$Y$. Moreover, for all $p\in\set{1,\dots,i+1}$ we have $N(r_p)\cap X_j =\emptyset$ in turn implying that $u_0\neq r_p$. On the other hand, for all $p\in\set{i+2,\dots,k-1}$ we have $\alpha \left( H_j \backslash N(r_p) \right)< k-1$, thus, $u_0\neq r_p$. This leaves only one possible option of $u_0=\ell_{i}$. The case \(j=2i\) follows a symmetric argument. See Figure~\ref{fig:H_3}. 

    Recall that all $k-1$ elements of $Y \backslash \set{u_0}$ must be in $H_j$. It is easy to see that when $j=2i-1$ then $H_j \backslash N(\ell_{i+1})$ has exactly one independent set of size $k-1$, namely: $\set{v_{N-(k-i-2)}, \dots,v_N,b_1\dots,b_i}$, symmetrically when $j=2i$ then $H_j \backslash N(r_{i+1})$ has exactly one independent set of size $k-1$, namely: $\set{a_i, \dots,a_1,v_1,\dots,v_{k-1}}$. This proves that $Y=C_{j+1}$.

    \begin{figure}[!h]
        \centering
        \includegraphics{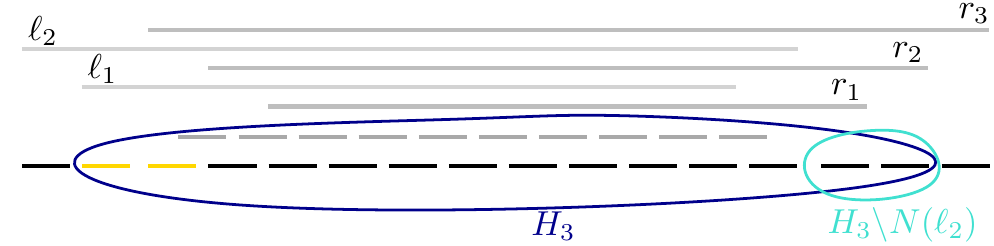}
        \caption{The set $H_3$ for $G_{6,3}$. We interpret $H_3 \backslash N(u)$ as the space where $k-1$ tokens can "hide". All base neighbours of $\ell_1,r_1$, and $r_2$ are in $H_3$. Also, $\alpha \left( H_3 \backslash N(r_3) \right)=1<2$. This gives $u_0 = \ell_2$.}
        \label{fig:H_3}
    \end{figure}

    Let us now prove that for a fixed $i\in\set{1,\dots,k-1}$ every path from $C_{2i-1}$ to $C_{2i}$ in $R_k(G_{m,k})$ is of length at least $(k-2i-1) \cdot m$. 

    Fix the shortest reconfiguration sequence from $C_{2i-1}$ to $C_{2i}$. As tokens do not interchange their relative positions, note that tokens starting at the positions $(v_1, \dots, v_{k-2i})$ in $C_{2i-1}$ must finish at the positions $( v_{N-(k-2i-1)}, \dots, v_N )$ in $C_{2i}$. We call these tokens \emph{heavy}. Their left to right ordinal numbers are $i\dots,k-i$, and there are exactly $s:=k-2i-1$ of them.

    We prove that a heavy token cannot use any of the long intervals during the reconfiguration. By the first part of the proof we know that on the shortest path from $C_{2i-1}$ to $C_{2i}$ only base intervals from $H_{2i-1}$ can be used. For each long interval $u$ we define $H_u^\ell$ as a graph induced by all intervals $v$ in $H_{2i-1}\backslash N(u)$ completely to the left of $u$. Analogously, define $H_u^r$ as the graph induced by all $v\in H_{2i-1}\backslash N(u)$ completely to the right to $u$. Finally, put
    $$n_\ell(u) = \alpha(H_u^\ell) \ \text{and} \ n_r(u) = \alpha(H_u^r).$$
    Assume that a heavy token uses a fixed long interval $w$ on the path from $C_{2i-1}$ to $C_{2i}$. Armed with the knowledge of the ordinal numbers of the heavy tokens, we see that: $n_\ell(w) \ge i-1$ and $n_r(w) \ge i$. Elementary computation shows that for every long interval $u$ either $n_\ell(u) < i-1$ or $n_r(u) < i$, which proves that no such long interval $w$ exists.

    As heavy tokens cannot use long intervals, each of them has to use base and path intervals forcing it to make at least $2(N-s+1) \ge m$ steps. Therefore, we need at least $s\cdot m$ steps in the path. 

    Summing up all required steps, we conclude, that every path from $I$ to $J$ in $R_k(G_{m,k})$ has length at least $\frac{k^2}{4}\cdot m$.

\end{proof}

\section{Hardness result for incomparability graphs}

In this section, we present a simple reduction showing that \isr\ is PSPACE-hard on incomparability graphs in general. Note that interval graphs are incomparability graphs of interval orders.
The proof exhibits a reduction from \(H\)\textsf{-Word Reachability} defined in~\cite{wrochnaReconf}. For the readers' convenience we state the definition of this problem here. If \(H\) is a digraph (possibly with loops) and \(a = a_1 a_2 \dots a_n \in V(H)^*\) then \(a\) is an \(H\)\emph{-word}, if for any \(i \in \set{1, \dots, n - 1}\) we have \(a_i a_{i + 1} \in E(H)\).
In the \(H\)\textsf{-Word Reachability} we are given two \(H\)-words of the same length \(a\) and \(b\), and the question is whether one can transform \(a\) into \(b\) by changing one letter at a time in such a way that each intermediate word is an \(H\)-word.

\begin{theorem}[\cite{wrochnaReconf}, Theorem 3]\label{thm:wrochnaHardness}
    There exists a digraph \(H\) for which the \(H\)\textsf{-Word Reachability} is PSPACE-complete.
\end{theorem}

\begin{theorem}\label{thm:hardness}
    There exists a constant \(w \in \mathbb{N}\), such that \isr\ is PSPACE-hard on incomparability graphs of posets of width at most \(w\).
\end{theorem}
\begin{proof}
    We demonstrate a reduction from \(H\)\textsf{-Word Reachability} for arbitrary \(H\); the result will follow from \cref{thm:wrochnaHardness}.

    Fix an instance of \(H\)\textsf{-Word Reachability} consisting of two \(H\)-words \(a\) and \(b\) of equal length $n$. We will construct a poset of width at most \(2 \card{V(H)}\), and two independent sets \(A, B\) in its incomparability graph, such that $A$ is reconfigurable to \(B\) if and only if our starting instance is a yes instance of \(H\)\textsf{-Word Reachability}.
    Define the poset \(P_n(H)\) as \((V(H) \times \set{1,\dots,n}, \prec)\) where $\prec$ is defined as follows:
    \[
        (x, i) \prec (y, j)  \iff \paren*{ j = i + 1 \text{ and } xy \in E(H)} \text{ or } \paren*{ j > i + 1 }\text{.}
    \]

    By the definition of \(\prec\) each set of the form \(V(H) \times \set{i}\) is an antichain, thus for any chain \(C\) in \(P_n(H)\) of cardinality \(n\) and any $i\in\set{1,\dots,n}$, we have $|(V(H)\times\set{i}) \cap C| = 1$. Therefore, any chain \(C\) of cardinality $n$, can be written as \(C = \set{(x_1, 1), (x_2, 2), \dots, (x_n, n)}\). 
    Observe that for each \(i\in \set{1,\dots,n-1}\) we have \((x_i, i) \prec (x_{i + 1}, i + 1) \iff x_i x_{i + 1} \in E(H)\). This implies that the first coordinates \(x_1 x_2 \dots x_n\) of the elements of chain $C$ form an \(H\)-word. 
    Conversely, given an \(H\)-word consisting of \(n\) letters \(y_1 y_2 \dots y_n\) the set \(\set{(y_1, 1), (y_2, 2), \dots, (y_n, n)}\) is a chain of cardinality \(n\) in \(P_n(H)\).
    It follows that a word \(x_1 x_2 \dots x_n\) is an \(H\)-word if and only if \(\set{(x_1, 1), \dots, (x_n, n)}\) is an independent set in the incomparability graph \(\text{Inc}(P_n(H))\).

    Let \(a = a_1 a_2 \dots a_n\) and \(b = b_1 b_2 \dots b_n\) be the two given $H$-words of length $n$. We define \(A = \set{(a_1, 1), (a_2, 2), \dots, (a_n, n)}\) and \(B = \set{(b_1, 1), (b_2, 2), \dots, (b_n, n)}\). These are two independent sets in \(\text{Inc}(P_n)\). Using the fact that for each $i \in \set{1,\dots,n}$ the set \(V(H) \times \set{i}\) is a clique in \(\text{Inc}(P_n(H))\), we infer that each edge in \(R_n(P_n)\) corresponds to a move of the form \(((x, i), (y, i))\) for some $i \in \set{1,\dots,n}$. Thus \(A\) is reconfigurable into \(B\) if and only if one can transform \(a\) into \(b\) one letter at a time keeping each intermediate word an \(H\)-word.

    All that remains is to observe that we can construct the incomparability graph of \(P_n(H)\) together with the sets \(A\) and \(B\) for a fixed \(H\) in logarithmic space, and that the width of \(P_n(H)\) is always at most \(2\card{V(H)}\).
\end{proof}

\bibliographystyle{abbrv}
\bibliography{bibliography}
\end{document}